\documentclass[11pt]{article}
\usepackage{amsthm,amssymb,amsmath,graphicx,cite}
\usepackage[percent]{overpic}
\usepackage{color}

\newcommand{\ip}[2]{\left\langle #1 , #2 \right\rangle}    
\newcommand{\V}{{\mathcal V}}
\newcommand{\R}{{\mathbb R}}
\renewcommand{\S}{{\mathbb S}}

\newcommand{\sym}{{\mathrm{sym}}}
\newcommand{\Sym}{{\mathrm{Sym}}}

\newcommand{\vertiii}[1]{{\vert\kern-0.25ex\vert\kern-0.25ex\vert #1 
    \vert\kern-0.25ex\vert\kern-0.25ex\vert}}

\DeclareMathOperator{\dist}{dist}

\DeclareMathOperator{\interior}{int}
\DeclareMathOperator{\Image}{Im}

\newtheorem{theorem}{Theorem}
\newtheorem{proposition}{Proposition}

\newtheorem{corollary}{Corollary}

\newtheorem{example}{Example}
\def\transp{^{\text{\sf T}}}
\newcommand{\matr}[1]{\begin{bmatrix} #1 \end{bmatrix}}    

\newcommand{\I}{\mathcal I}

\title{A data-independent distance to infeasibility for linear conic systems}
\author{Javier Pe\~na\thanks{Tepper School of Business,
Carnegie Mellon University, USA, {\tt jfp@andrew.cmu.edu}} \and Vera Roshchina\thanks{School of Mathematics and Statistics, UNSW Sydney, Australia, {\tt v.roshchina@unsw.edu.au}}}

\begin{document}

\maketitle

\begin{abstract} We offer a unified treatment of distinct measures of well-posedness for homogeneous conic systems.  To  that end, we introduce a distance to infeasibility based entirely on geometric considerations of the elements defining the conic system.  
Our approach sheds new light on and connects several well-known condition measures for conic systems, including {\em Renegar's} distance to infeasibility, the {\em Grassmannian} condition measure, a measure of the {\em most interior} solution, and other geometric measures of {\em symmetry} and of {\em depth} of the conic system.

\medskip


\medskip


\end{abstract}

\newpage

\section{Introduction}

The focus of this work is the geometric interpretation and coherent unified treatment of measures of well-posedness for homogeneous conic problems. We relate these different measures via a new geometric notion of a distance to infeasibility. 

The development of condition measures in optimization
was pioneered by Renegar \cite{isitpossible,Rene95a,Rene95} and has been further advanced by a number of scholars.
Condition measures provide a fundamental tool to study various aspects of problems such as the behavior of solutions, robustness and sensitivity analysis~\cite{canovas16,NuneF98,Pena00, Rene94}, and performance of algorithms~\cite{BellFV09,EpelF00,Freu04,OrdoF03,PenaS16,Rene95}.  
Renegar's condition number for conic programming is defined in the spirit of the classical matrix condition number of linear algebra, and is explicitly expressed in terms of the distance to infeasibility, that is, the smallest perturbation on the data defining a problem instance that renders the problem infeasible~\cite{Rene95a,Rene95}. By construction, Renegar's condition number is inherently data-dependent.  A number of alternative approaches for condition measures are defined in terms of the intrinsic geometry of the problem  and independently of its data representation.  Condition measures of this kind include the {\em symmetry} measure studied by Belloni and Freund~\cite{BellF08}, the {\em sigma} measure used by Ye~\cite{Ye94}, and the {\em Grassmannian} measure introduced by Amelunxen and B\"urgisser~\cite{AmelB12}  which extends a construction of Belloni and Freund~\cite{BellF09a}.  In addition, other condition measures such as the ones used by Goffin~\cite{Goff80}, Cheung and Cucker~\cite{CheuC01}, Cheung et al.~\cite{CheuCP10}, and by Pe\~na and Soheili~\cite{PenaS16} are defined in terms of {\em most interior solutions}.

At a fundamental level, a main goal of a condition measure is to capture the  ``difficulty'' or ``tractability'' of a problem.
The variety of condition measures for optimization reflects the challenges in achieving this goal.  This is not surprising since the actual difficulty of a problem generally depends on the representation and solution methods available.  A particular condition number would typically yield overly conservative bounds on quantities of interest, such as geometric properties of the solution set or the convergence rate of an algorithm, if those quantities are invariant under some transformations but the condition number is not.  The development of various kinds of condition measures can be attributed to this tension between condition measures and invariance under different kinds of transformations.  The central goal of this paper is to shed new light on and relieve this tension.  To achieve that goal, we focus our attention on the following three minimal components of a conic system: the convex cone and linear subspace that define the conic system, and an underlying norm in the ambient space.  Our approach enables us to uncover and highlight common ideas and differences underlying the most popular condition measures for conic systems, reveals some extensions, and establishes close relationships among them.  

We define a {\em data-independent} distance to infeasibility that depends solely on the above three minimal components (cone, linear subspace, and norm).  In the particular case when the norm is the Euclidean norm, the data-independent distance to infeasibility coincides with the Grassmannian condition measure introduced by Belloni and Freund~\cite{BellF09a} and further extended by Amelunxen and B\"urgisser~\cite{AmelB12}.  The latter concept in turn is closely related to the {\em angular separation criterion} proposed by Flinth~\cite{Flin16} to formalize stability and robustness properties that lie at the heart of sparse signal recovery~\cite{AmelLMT14,Cahi16, CahiM14, CandRT06,CandP10,ChanRPW12}.  
However, we should emphasize that our construction of a data-independent distance to infeasiblity applies to any norm (not necessarily Euclidean).  The flexibility of working with non-Euclidean norms and more general non-Euclidean geometries has led to major advances in optimization, particularly in first-order algorithms~\cite{SraNW12,Tebo18}.  Non-Euclidean norms typically fit the geometry of the problem more naturally, prime examples being the one-norm for the non-negative orthant and the nuclear norm for the positive semidefinite cone.  The flexibility in the choice of norms is a main novelty in our construction and a key feature for most of our developments.  In particular, the flexibility in the choice of norm enables us to establish a number of interesting connections with other geometric properties of the conic systems such as a measure of {\em symmetry} and a measure of {\em depth} of the conic system.  Our derivation of these connections in turn provides new interesting insight into some canonical {\em induced eigenvalue mappings} and {\em induced norms} associated to the cone defining the conic system.  The latter objects are tied to the structural properties of the cone and play central roles in optimization models and algorithms.  Two canonical examples of induced norms and their duals are the infinity and one norms in $\R^n$ induced by the non-negative orthant $\R^n_+$, and the operator norm and nuclear norm in $\S^n$ induced by the positive semidefinite cone $\S^n_+$.  

Our developments highlight the tradeoffs of different notions of conditioning.  That kind of insight in turn suggests {\em preconditioning} and {\em reconditioning} techniques to improve the well-posedness of a problem.
The former type of technique (preconditioning) can be applied to preprocess the data representing a problem so that the problem is ``better posed''.  
 Although preprocessing procedures are routinely used by optimization solvers, they are not always founded on a formal theory.  The latter type of technique (reconditioning) can be seen as an adaptive variant of preconditioning that transforms a problem as new information is gathered.  This type of reconditioning technique underlies a variety of {\em rescaling algorithms} such as the rescaled perceptron algorithm of Dunagan and Vempala~\cite{DunaV06}, the more recent Chubanov's rescaling and projection algorithm~\cite{Chub15} and a number of subsequent related developments~\cite{BellFV09,KitaT18,LourKMT16,PenaS16,Roos18}.  Most of these algorithms are based on alternating between a {\em basic procedure} and a {\em rescaling procedure}.  The  basic procedure attempts to solve the problem and succeeds if the problem is well conditioned.  If it does not succeed, then it provides guidance for the rescaling procedure to recondition the problem so that the basic procedure can be applied again to an equivalent but better conditioned problem.  A similar alternating principle  also underlies a variety of increasingly popular {\em restarting techniques} for first-order algorithms for convex optimization~\cite{RoulD17}.

We focus on the feasibility problems that can be represented as the intersection of a closed convex cone with a linear subspace.
Feasibility problems of this form are pervasive in optimization.  The  constraints of linear, semidefinite, and more general conic programming problems are written explicitly as the intersection of a (structured) convex cone with a linear (or, more generally, affine) subspace. The fundamental signal recovery property in compressed sensing can be stated precisely as the infeasibility of a homogeneous conic system for a suitable choice of a cone and linear subspace as explained in~\cite{AmelLMT14,ChanRPW12}.
 Our {\em data-independent} distance to infeasibility is a measure of proximity between the orthogonal complement of this linear subspace and the dual cone.  This distance depends only on the norm, cone, and linear subspace.  Specific choices of norms lead to interpretations of this distance as the Grassmannian measure \cite{AmelB12} as well as a measure of the most interior solution \cite{CheuCP10}.  Our approach also yields neat two-way bounds between the sigma measure \cite{Ye94} and symmetry measure \cite{BellF08,BellF09b} in terms of this geometric distance.
Our work is inspired by~\cite{AmelB12}, and is similar in spirit to an abstract setting of convex processes \cite[Section~5.4]{BorwL06} (also see \cite{DontLR03}). For a more general take on condition numbers for unstructured optimization problems and for an overview of recent developments we refer the reader to \cite{zolezzi}.

The main sections of the paper are organized as follows.  
We begin by defining our data-independent distance to infeasibility in Section~\ref{sec.measures}, where we also show that it coincides with the Grassmannian distance of \cite{AmelB12} for the Euclidean norm. In Section~\ref{sec.renegar} we discuss Renegar's distance to infeasibility and show in  Theorem~\ref{thm.renegar}  that the ratio of the geometric distance to infeasibility and Renegar's distance is sandwiched between the reciprocal of the norm of the matrix and the norm of its set-valued inverse, hence  extending~\cite[Theorem 1.4]{AmelB12} to general norms. In Section~\ref{sec.induced.norm} we show that the cone induced norm leads to the interpretation of the distance to infeasibility in terms of the most interior solution (Proposition~\ref{prop.eigenvalue}). We also provide further interpretation as eigenvalue estimates for the cone of positive semidefinite matrices and for the nonnegative orthant. 

In Section~\ref{sec:sigma} we propose an extension of the sigma measure of Ye   and establish bounds relating the sigma measure and the distance to infeasibility (Proposition~\ref{corol.nu.sigma}).  
Section~\ref{sec.symmetry.meas} relates our distance infeasibility and the sigma measure to the symmetry measure used by Belloni and Freund via neat symmetric bounds in Theorem~\ref{thm.symmetry} and Corollary~\ref{corol.symmetry}.  Finally, Section~\ref{sec.extend} describes extensions of our main developments via a more flexible choice of norms.

\section{Data-independent distance to infeasibility}
\label{sec.measures}

Let $E$ be a finite dimensional real vector space with an inner product $\ip{\cdot}{\cdot}$, endowed with a (possibly non-Euclidean) norm $\|\cdot\|$. Recall that the dual norm $\|\cdot\|^*$ is defined for $u \in E$ as
$$\|u\|^* := \max_{\|x\|=1}\ip{u}{x}.$$
Notice that by construction, the following  {\em H\"older's inequality} holds for all $u,x\in E$
\begin{equation}\label{eq.holder}
|\ip{u}{x}| \le \|u\|^* \cdot \|x\|.
\end{equation}

Let $K\subseteq E$ be a closed convex cone.  Given a linear subspace $L\subseteq E$, consider the feasibility problem
\begin{equation}\label{primal}
\text{ find} \; x\in L\cap K \setminus\{0\}
\end{equation}
and its alternative
\begin{equation}\label{dual}
\text{ find }  u \in L^\perp \cap K^* \setminus\{0\}.
\end{equation}
Here $K^*$ denotes the {\em dual cone} of $K$, that is,
\[
K^*:=\{u\in E: \ip{u}{x} \ge 0  \; \forall x\in K\},
\]
and $L^\perp$ is the {\em orthogonal complement} of the linear subspace $L$,
\[
L^\perp:= \{u\in E: \ip{u}{x} = 0  \; \forall x\in L\}.
\]

In what follows we assume that $K\subseteq E$ is a closed convex cone that is also \emph{regular},  that is,  $\interior(K) \neq \emptyset$ and $K$ contains no lines. In our analysis the cone $K$ is fixed, and the linear subspace $L$ is treated as the problem instance. This is a standard approach that stems from the real-world models, where the cone is a fixed object with well-known structure that encodes the model's structure (for instance, the nonnegative orthant, the cone of positive semidefinite matrices, copositive or hyperbolicity cone), and the problem instance is encoded via the coefficients of a linear system that in our case corresponds to the linear subspace.  

Observe that \eqref{primal} and \eqref{dual} are {\em alternative systems:} one of them has a strictly feasible solution if and only if the other one is infeasible. When neither problem is strictly feasible, they both are {\em ill-posed:} each problem becomes infeasible for arbitrarily small perturbations of the linear subspace.

The main object of this paper is the following data-independent {\em distance to infeasibility} of~\eqref{primal}:
\begin{equation}\label{eq:defcondition}
\nu(L):=\min_{u\in K^*, y\in L^\perp\atop\|u\|^*=1} \|u-y\|^*.
\end{equation}
Observe that $\nu(L)\geq 0$ and  $L\cap \interior(K) \ne \emptyset$ if and only if $\nu(L) >0$. Furthermore, $\nu(L)$ is the distance between the space $L^\perp$ and the set $\{u\in K^*: \|u\|^* = 1\}$, or equivalently between $L^\perp$ and $\{u\in K^\circ: \|u\|^* = 1\}$ for $K^\circ = -K^*$, as illustrated in Figure~\ref{fig:definition}.
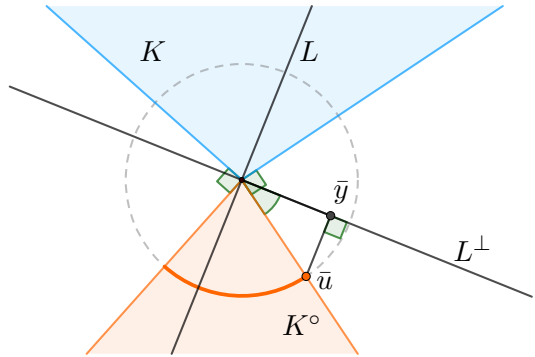
\begin{figure}[ht]
	\centering
\begin{overpic}[width=0.6\textwidth
]{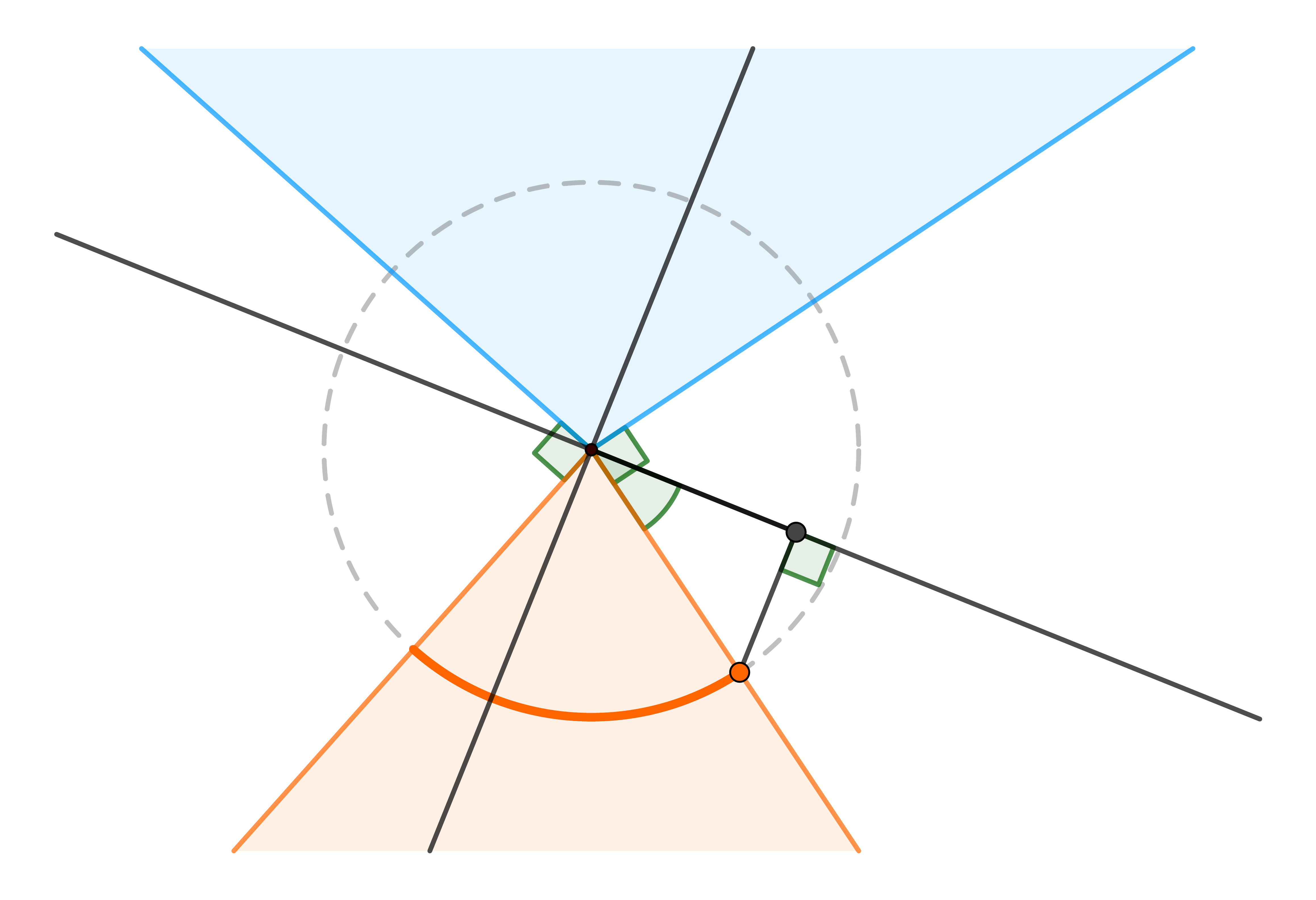}
 \put (27,55) {$K$}
 \put (52,7) {$K^\circ$}
 \put (55,55) {$L$}
 \put (82,20) {$L^\perp$}
 \put (61,31) {$\bar y$}
 \put (58,15) {$\bar u$}
\end{overpic}
	\caption{Illustration of $\nu(L)$ when $\nu(L)>0$. Here  $\bar u$ and $\bar y$ denote the points attaining the minimum in \eqref{eq:defcondition}, so that $\nu(L) = \|\bar u-\bar y\|^*$.}
	\label{fig:definition}
\end{figure}
Since both~\eqref{primal} and~\eqref{dual} are defined via a cone and a linear subspace, there is a natural symmetric version of distance to infeasibility for \eqref{dual} obtained by replacing $K^*$, $L^\perp$ and $\|\cdot \|^*$ in \eqref{eq:defcondition} with their primal counterparts.

When the norm $\|\cdot \|$ is Euclidean, that is, $\|v\| = \|v\|^*= \|v\|_2 = \sqrt{\ip{v}{v}}$, the distance to infeasibility \eqref{eq:defcondition} coincides with the {\em Grassmann distance to ill-posedness} defined by Amelunxen and B\"urgisser~\cite{AmelB12}. To see this, first observe that the Euclidean norm is naturally related to angles.  Given $x,y\in E\setminus\{0\}$ let $\angle(x,y) := \arccos \frac{\ip{x}{y}}{\|x\|_2\|y\|_2} \in [0,\pi]$.  Given a linear subspace $L\subseteq E$ and a closed convex cone $C\subseteq E$, let
\[
\angle(L,C) := \min\{\angle(x,v): x \in L\setminus\{0\}, \; v \in C\setminus\{0\}\} \in [0,\pi/2].
\]

\begin{proposition}  \label{prop.euclidean} If $\|\cdot\| = \|\cdot\|_2$ then 
\[
\nu(L) = \sin\angle(L^\perp,K^*).
\]
\end{proposition}
\begin{proof}
Since $\angle(L^\perp,K^*) \in [0,\pi/2]$ we have
\[
\sin\angle(L^\perp,K^*) = \min_{u\in K^*,y\in L^\perp \atop u,y\ne 0}  \sin \angle(y,u) 
= \min_{u\in K^*,y\in L^\perp\atop \|u\|_2=1}  \|u-y\|_2 
= \nu(L).
\]
\end{proof}

Proposition~\ref{prop.euclidean} and~\cite[Proposition 1.6]{AmelB12} imply that when $\|\cdot\| = \|\cdot\|_2$ the distance to infeasibility $\nu(L)$ matches the Grassmann distance to ill-posedness of~\cite{AmelB12}.  The flexibility in the choice of norm in $E$ is a main feature in our construction of $\nu(L)$ as some norms are naturally more compatible with the cone.   Suitable choice of norms generally yield sharper results in various kinds of analyses.  In particular, in condition-based complexity estimates an appropriately selected norm typically leads to tighter bounds. The articles~\cite{CCPMultifold,PenaS16} touch upon this subject, and consistently in \cite{canovas16} a sup-norm is deemed a convenient choice for the perturbation analysis of linear programming problems.

\medskip

We will rely on the following characterization of $\nu(L)$.

\begin{proposition}\label{prop.nu} 
If $L$ is a linear subspace of $E$ and $L\cap \interior(K)\ne \emptyset$ then the distance to infeasibility \eqref{eq:defcondition} can be equivalently characterized as
\[
\nu(L) =
\min_{u\in K^* \atop \|u\|^*=1} \max_{x\in L\atop \|x\|\le 1} \ip{u}{x}.
\]
\end{proposition}
\begin{proof}
The construction of the dual norm and 
Sion's minimax theorem \cite{Sion}  imply that for all $u\in E$ 
\begin{equation}\label{nu.eq.1}
\min_{y\in L^\perp} \|u-y\|^* = \min_{y\in L^\perp}\max_{x\in E\atop \|x\| \le 1}\ip{u-y}{x}=
\max_{x\in E\atop \|x\| \le 1}\min_{y\in L^\perp}\ip{u-y}{x}.
\end{equation}
Next, observe that for all $x\in E$
\begin{equation}\label{nu.eq.2}
\min_{y\in L^\perp} \ip{-y}{x} = \left\{\begin{array}{lll}-\infty & \text{ if } x\not\in L\\
0 & \text{ if } x\in L.
\end{array}\right.
\end{equation}
Putting~\eqref{nu.eq.1} and~\eqref{nu.eq.2} together we get
\[
\min_{y\in L^\perp} \|u-y\|^* = \max_{x\in L\atop \|x\| \le 1}\ip{u}{x}.
\]
Therefore $\nu(L) = \displaystyle\min_{u\in K^*, y\in L^\perp \atop \|u\|^*=1} \|u-y\|^*=
\min_{u\in K^* \atop \|u\|^*=1} \max_{x\in L\atop \|x\|\le 1} \ip{u}{x}.$
\end{proof}

We conclude this section by briefly noting two natural variants $\overline \nu(L)$ and $\mathcal V(L)$ of $\nu(L)$.  To ease our exposition, we defer a more detailed discussion of these variants to Section~\ref{sec.extend}.  
The variant $\overline \nu(L)$ is obtained by normalizing $y$ instead of $u$, that is,
\[
\overline\nu(L):=\min_{\substack{u\in K^*, y\in L^\perp\\\|y\|^*=1}} \|u-y\|^*.
\]
The second variant incorporates additional flexibility by allowing the use of different norms in the normalization of $u$ and in the difference $y-u$, that is,
\[
\V(L):=\min_{\substack{u\in K^*, y\in L^\perp\\\|u\|^*=1}} \vertiii{u-y}^*.
\]
where $\vertiii{\cdot}$ is an additional norm in $E$.

\section{Renegar's distance to infeasibility}
\label{sec.renegar}
We next relate the condition measure $\nu(\cdot)$ with the classical Renegar's distance to infeasibility.  A key conceptual difference between Renegar's approach and the approach used above is that Renegar~\cite{Rene95a,Rene95}  considers conic feasibility problems where the linear spaces $L$ and $L^\perp$ are explicitly defined as the image and the kernel of the adjoint of some linear mapping. 

For  a linear mapping $A:F\rightarrow E$ between two normed real vector spaces $F$ and $E$ consider the conic systems \eqref{primal} and \eqref{dual} defined by taking $L=\Image(A)$. 
These two conic systems can respectively be written as 
\begin{equation}\label{renegar.primal}
Ax \in K \setminus\{0\}
\end{equation}
and \begin{equation}\label{renegar.dual} 
 A^*w = 0, \; w\in K^*  \setminus\{0\}.
\end{equation}
Here $A^*:E\rightarrow F$ denotes the {\em adjoint} operator of $A$, that is, the linear mapping satisfying $\ip{y}{Aw} = \ip{A^*y}{w}$ for all $y\in E, w\in F.$

Let $\mathcal L(F,E)$ denote the set of linear mappings from $F$ to $E$. Endow $\mathcal L(F,E)$ with the operator norm, that is,
\[
\|A\|:=\max_{w\in F\atop |w| \le 1} \|Aw\|,
\]
where $|\cdot|$ is the norm in $F$.

Let $A\in \mathcal L(F,E)$ be such that \eqref{renegar.primal} is feasible.  The {\em distance to infeasibility} of \eqref{renegar.primal} is defined as
\begin{align*}
\dist(A,\I) &:= \inf\left\{\|A - \tilde A\|: \tilde Ax \in K\setminus\{0\} \; \text{ is infeasible}\right\}
\\ 
& \; = \inf\left\{\|A - \tilde A\|: \tilde A^* w = 0 \; \text{ for some } w \in K^*\setminus \{0\}\right\}.
\end{align*}
Observe that~\eqref{renegar.primal} is strictly feasible if and only if $\dist(A,\I) > 0$.

Given $A  \in \mathcal L(F,E),$ let  $A^{-1}: \Image(A)\rightrightarrows F$ be the set-valued mapping defined via $x \mapsto \{w\in F: Aw = x\}$ and 
\[
\|A^{-1}\|:=\max_{x\in \Image(A)\atop \|x\| \le 1} \min_{w\in A^{-1}(x)} |w|.
\]

The following result is inspired by and extends~\cite[Theorem 1.4]{AmelB12} and~\cite[Theorem 3.1]{BellF09a}. More precisely,~\cite[Theorem 1.4]{AmelB12} and~\cite[Theorem 3.1]{BellF09a} coincide with Theorem~\ref{thm.renegar} in the special case $\|\cdot\| = \|\cdot\|_2$.

\begin{theorem} 
\label{thm.renegar}
Let $A \in \mathcal L(F,E)$ be such that \eqref{renegar.primal} is strictly feasible and let $L:=\Image(A)$.  Then
\[
\frac{1}{\|A\|} \le \frac{\nu(L)}{\dist(A,\I)} \le \|A^{-1}\|.
\]
\end{theorem}
\begin{proof}
First, we prove $\dist(A,\I) \le \nu(L) \|A\|.$  To that end, let $\bar u \in K^*$ be such that $\|\bar u\|^* =1$ and 
$
\nu(L) = \displaystyle\max_{x\in L \atop \|x\|\le 1} \ip{\bar u}{x}
$ as in Proposition~\ref{prop.nu}.
Then
\begin{equation}\label{eq:relAstar}
 |A^*\bar u|^* = \max_{w\in F \atop |w|\le 1} \ip{\bar u}{Aw} \le \nu(L) \|A\|.
\end{equation}
Let $\bar v\in E$ be such that $\|\bar v\| = 1$ and $\ip{\bar u}{\bar v} = \|\bar u\|^* =1$.  
Now construct $\Delta A: F \rightarrow E$ as follows
\[
\Delta A(w):= - \ip{A^* \bar u}{w} \bar v.
\]
Observe that $\|\Delta A\| = |A^*\bar u|^* \cdot \|\bar v\| \le \nu(L)\|A\|$ (by \eqref{eq:relAstar}) and $\Delta A^*: E \rightarrow F$ is defined by
\[
\Delta A^*(y) = - \ip{y}{\bar v} A^*\bar u.
\]
In particular $(A+ \Delta A)^*\bar u = A^* \bar u - \ip{\bar u}{\bar v} A^*\bar u = 0$ and $\bar u \in K^*\setminus \{0\}.$  Therefore
\[
\dist(A,\I) \le \|\Delta A\| \le \nu(L)\|A\|.
\]
Next, we prove   
$\nu(L)\le \| A^{-1}\|\dist(A,\I) $. To that end, suppose  
$\tilde A\in \mathcal L(F,E)$ is such that $\ker(\tilde A^*) \cap K^*\setminus\{0\} \ne \emptyset.$   Let $\bar u \in K^*$ be such that $\|\bar u \|^* = 1$ and $\tilde A^*(\bar u) = 0$.  From the construction of $\|A^{-1}\|$, it follows that for all $x\in L=\Image(A)$ there exists $w\in A^{-1}(x)$ such that $|w| \le \|A^{-1}\| \cdot \|x\|$.  Since $\bar u \in K^*$ and $\|\bar u \|^*=1,$ Proposition~\ref{prop.nu} implies that
\[
\nu(L) \le \max_{x\in \Image(A) \atop \|x\| \le 1} \ip{\bar u}{x} \le 
\max_{w\in F \atop |w| \le \|A^{-1}\|} \ip{\bar u}{Aw} 
= \|A^{-1}\|\cdot |A^*\bar u|^*.
\]
Next, observe that $|A^*\bar u|^* =  |(\tilde A-A)^* \bar u|^* \le \|\tilde A- A\|$ because $\|\bar u\|^* =1$ and $\tilde A^* \bar u = 0$.  Thus $\nu(L) \le \|A^{-1}\|\cdot \|\tilde A - A\|$.
Since this holds for all $\tilde A\in \mathcal L(F,E)$ such that $\ker(\tilde A^*) \cap K^*\setminus\{0\} \ne \emptyset$ it follows that 
\[
\nu(L) \le \|A^{-1}\| \dist(A,\I).
\]
\end{proof}

 Proposition~\ref{thm.renegar.again} in Section~\ref{sec.extend} below discusses an analogue of Theorem~\ref{thm.renegar} for the case when $L = \ker(A)$ for some linear map $A:E\rightarrow F$.  We defer that discussion to Section~\ref{sec.extend} because Proposition~\ref{thm.renegar.again} relies on the variant $\overline \nu(L)$ of $\nu(L)$.

\section{Induced norm and induced eigenvalue mappings}
\label{sec.induced.norm}
In addition to our assumption that $K \subseteq E$ is a regular closed convex cone, throughout the sequel we assume that $e\in \interior(K)$ is fixed.  We next describe a norm $\|\cdot\|_e$ in $E$ and a mapping $\lambda_{e}:E\rightarrow \R$ induced by the pair $(K,e)$.  These norm and mapping yield a natural alternative interpretation of $\nu(L)$ as a measure of the {\em most interior} solution to the feasibility problem $x \in L\cap \interior(K)$ when this problem is feasible.  

Define the norm $\|\cdot\|_e$ in $E$ induced by $(K,e)$ as follows (see \cite{CCPMultifold})
\[
\|x\|_e:=\min\{\alpha \ge 0: x+\alpha e \in K, \; -x+\alpha e \in K\}.
\]
For the special case of the nonnegative orthant $\R^n_+$ this norm has a natural interpretation: it is easy to check that for $e = \matr{1 & \cdots & 1}\transp$ we obtain $\|\cdot \|_e = \|\cdot \|_{\infty}$. The geometric interpretation is shown in Figure~\ref{fig:orthant}.
\begin{figure}[ht]
	\centering
\begin{overpic}[width=0.45\textwidth,
]{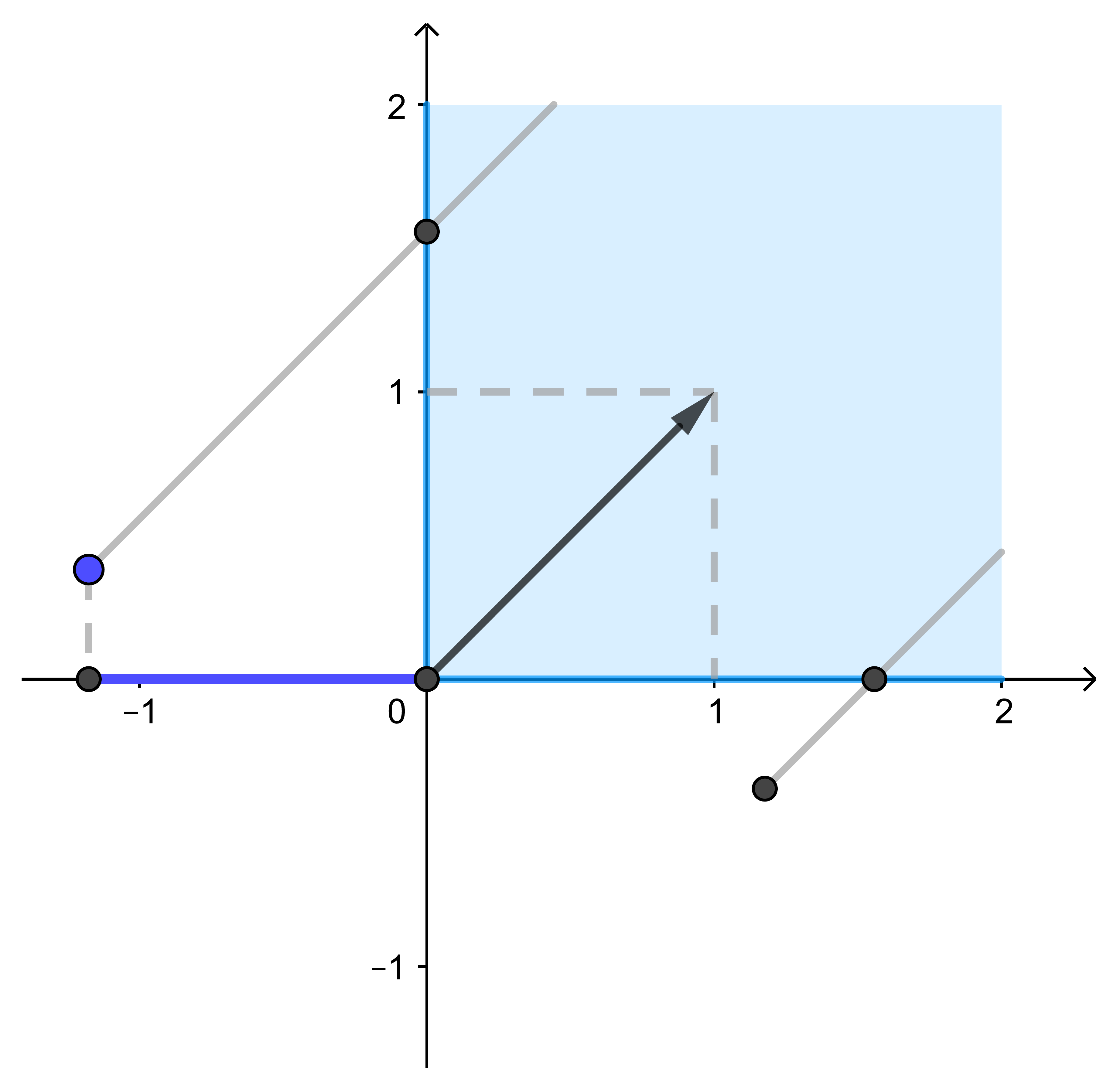}
 \put (12,45) {$x$}
 \put (70,25) {$-x$}
 \put (50,55) {$e$}
 \put (47,77) {$x+e\R_+$}
 \put (70,50) {$-x+e\R_+$}
 \put (8,25) {$\|x\|_e = \|x\|_\infty$}
\end{overpic}
\begin{overpic}[width=0.45\textwidth
]{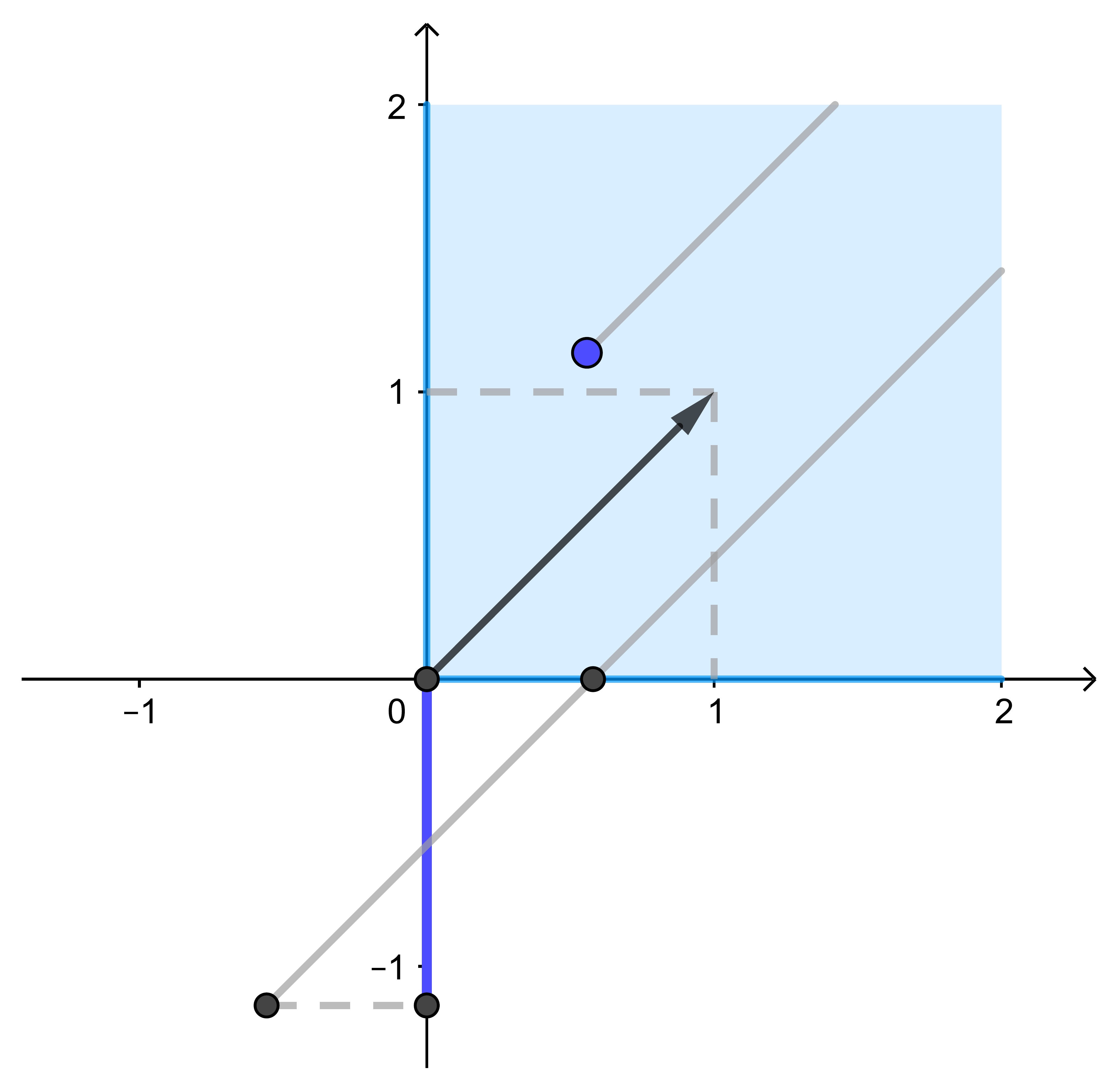}
 \put (45,68) {$x$}
 \put (12,10) {$-x$}
 \put (50,55) {$e$}
 \put (67,77) {$x+e\R_+$}
 \put (70,50) {$-x+e\R_+$}
 \put (42,14) {$\|x\|_e = \|x\|_\infty$}
\end{overpic}
	\caption{Induced norm for the nonnegative orthant.}
	\label{fig:orthant}
\end{figure}
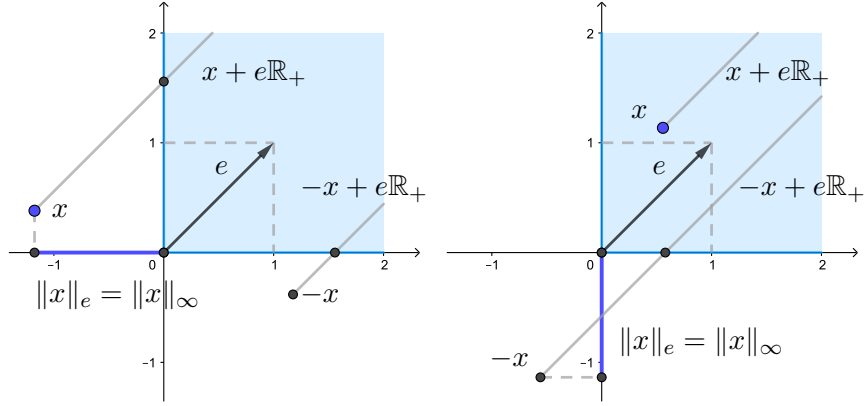
Define the {\em eigenvalue} mapping $\lambda_{e}:E \rightarrow \R$ induced by $(K,e)$ as follows
\[
\lambda_{e}(x):=\max\{t\in \R: x-te \in K\}.
\]
Observe that $x \in K \Leftrightarrow \lambda_{e}(x) \ge 0$ and $x \in \interior(K) \Leftrightarrow \lambda_{e}(x) > 0.$  Furthermore, observe that when $x \in K$
\[
\lambda_e(x) = \max\{r\ge 0: v\in E, \|v\|_e \le r \Rightarrow x+v\in K\}.
\]
Thus for $x\in K$, $\lambda_e(x)$ is a measure of how interior $x$ is in the cone $K$.  

It is easy to see that  $\|u\|_e^* = \ip{u}{e}$ for $u\in K^*$.  In analogy to the standard simplex, let
\[
\Delta(K^*,e):=\{u\in K^*: \|u\|_e^*=1\} = \{u\in K^*:\ip{u}{e}=1\}.
\]
It is also easy to see that the  eigenvalue mapping $\lambda_{e}$ has the following alternative expression
\[
\lambda_{e}(x) = \min_{u\in \Delta(K^*,e)} \ip{u}{x}.
\]
The next result readily follows from Proposition~\ref{prop.nu} and convex duality.

\begin{proposition}\label{prop.eigenvalue}
If $\|\cdot\| = \|\cdot\|_e$, then for any linear subspace $L\subseteq E$
\[
\nu(L)= \min_{u\in \Delta(K^*,e)} \max_{x\in L\atop \|x\| \le 1} \ip{x}{u} = 
\max_{x\in L\atop \|x\| \le 1} \min_{u\in \Delta(K^*,e)} \ip{x}{u} = 
\max_{x\in L\atop \|x\| \le 1}  \lambda_{e}(x). 
\]
\end{proposition}

Proposition~\ref{prop.eigenvalue} in particular implies that when $L \cap \interior(K) \ne \emptyset$ the quantity $\nu(L)$ can be seen as a measure of the  most interior point in $L\cap  \interior(K)$.
We next illustrate  Proposition~\ref{prop.eigenvalue} in two important cases.  The first case is $E=\R^n$ with the usual dot inner product, $K = \R^n_+$ and $e = \matr{1 & \cdots & 1}\transp \in \R^n_+$.  In this case $ \|\cdot\|_e= \|\cdot\|_\infty, \, \|\cdot\|_e^* = \|\cdot\|_1$, $(\R^n_+)^*=\R^n_+$ and $\Delta(\R^n_+,e)$ is the standard simplex 
$\Delta_{n-1}:=\{x\in \R^n_+:\sum_{i=1}^n x_i = 1\}$.   Thus $\lambda_{e}(x) = \displaystyle\min_{i=1,\dots,n} x_i$ and for $\|\cdot\| = \|\cdot\|_e$ we have
\begin{equation}\label{eqn.nu.lp}
\nu(L)= \max_{x\in L\atop \|x\|\le1} \min_{j=1,\dots,n} x_j.
\end{equation}

The second special case is $E=\S^n$ with the trace inner product, $K = \S^n_+$ and $e = I \in \interior(\S^n_+)$.  In this case $\|\cdot\|_e $ and $\|\cdot\|_e^*$ are respectively the operator norm and the nuclear norm in $\S^n$.  More precisely
\[
\|X\|_e = \max_{i=1,\dots,n} |\lambda_i(X)|, \; \|X\|_e^* = \sum_{i=1}^n |\lambda_i(X)|,
\]
where  $\lambda_i(X),\; i=1,\dots,n$ are the usual eigenvalues of $X$.  Furthermore, $(\S^n_+)^*=\S^n_+$ and $\Delta(\S^n_+,I)$ is the {\em spectraplex} $\{X\in \S^n_+: \sum_{i=1}^n \lambda_i(X) = 1\}$.   Thus $\lambda_{e}(x) = \min_{j=1,\dots,n} \lambda_j(X).$ In addition, in a nice analogy to \eqref{eqn.nu.lp}, for $\|\cdot\| = \|\cdot\|_e$ we have
\begin{equation}\label{eqn.nu.sdp}
\nu(L)=\max_{X\in L\atop \|X\| \le 1}  \min_{j=1,\dots,n}\lambda_{j}(X).
\end{equation}

\section{Sigma measure}\label{sec:sigma}

The induced eigenvalue function discussed in Section~\ref{sec.induced.norm} can be defined more broadly.   Given $v\in K\setminus\{0\}$ define $\lambda_v:E\rightarrow [-\infty,\infty)$ as follows 
\[
\lambda_v(x):=\max\{t: x- tv \in K\}.
\]
Define the {\em sigma} condition measure of a linear subspace $L\subseteq E$ as follows
\begin{equation}\label{def.sigma.feas}
\sigma(L):=\min_{v\in K\atop \|v\|=1}  \max_{x\in L \atop \|x\| \le 1}\lambda_v(x).
\end{equation}
The quantity $\sigma(L)$ can be interpreted as a measure of the depth of $L\cap K$ within $K$ along all directions $v\in K$.  Proposition~\ref{prop.eigenvalue} and Proposition~\ref{corol.nu.sigma}(c) below show that $\sigma(L)$ coincides with the measure $\nu(L)$ of the most interior point in $L\cap K$ when $\|\cdot\| = \|\cdot\|_e$.

The construction~\eqref{def.sigma.feas} of $\sigma(L)$  can be seen as a generalization of the sigma measure introduced by Ye~\cite{Ye94}.  Observe that $L \cap \interior(K) \ne \emptyset$ if and only if $\sigma(L) > 0$.
Furthermore, in this case Proposition~\ref{corol.nu.sigma} below shows that the quantities $\sigma(L)$ and $\nu(L)$ are  closely related. 
To that end, we rely on the following analogue of Proposition~\ref{prop.nu}. 

\begin{proposition}\label{prop.dual.sym}  
Let $L\subseteq E$ be a linear subspace. Then
\begin{equation}\label{eq.sigma.norm}
\sigma(L) = 
\min_{v\in K, y \in L^\perp, u\in K^* \atop \|v\|=1,\ip{u}{v}=1} \|u-y\|^*.
\end{equation}
\end{proposition}
\begin{proof} Assume $v\in K$ is fixed. The construction of $\lambda_v$ implies that
\begin{align}\label{eq.lambda}
\max_{x \in L \atop \|x\|\le 1} \lambda_v(x)
&= \max_{x \in L, t\in \R \atop \|x\|\le 1, x-tv\in K} t \notag \\ \notag&
= \max_{x \in L, t\in \R \atop \|x\|\le 1} \min_{u\in K^*} (t + \ip{u}{x-tv}) \\
&=  \min_{u\in K^*}\max_{x \in L, t\in \R \atop \|x\|\le 1} (t + \ip{u}{x-tv}) \\ \notag
& = \min_{u\in K^* \atop \ip{u}{v}=1} \max_{x\in L\atop \|x\|\le 1} \ip{u}{x} \\
& = \min_{u\in K^*,y\in L^\perp \atop \ip{u}{v}=1} \|u-y\|^*,\notag
\end{align}
where on the second line we used the von Neumann minimax theorem \cite{vNeumann} (also see \cite[Theorem 11.1]{Guler}), and the last step follows from the identity $\displaystyle
\max_{x\in L, \|x\|\le 1} \ip{u}{x}=\min_{y\in L^\perp}\|u-y\|^*$ established in the proof of Proposition~\ref{prop.nu}.  We thus get \eqref{eq.sigma.norm} by taking minimum in~\eqref{eq.lambda} over the set $\{v\in K: \|v\|=1\}$.
\end{proof}
\begin{proposition} 
\label{corol.nu.sigma} Let $L\subseteq E$  be a linear subspace such that $L\cap \interior(K)\ne \emptyset$.
\begin{description}
\item[(a)] For any norm $\|\cdot\|$ in $E$ the following holds
\[
1\le \min_{v\in K, u\in K^* \atop \|v\|=1,\ip{u}{v}=1} \|u\|^* \le \frac{\sigma(L)}{\nu(L)} \le 
 \frac{1}{
 \displaystyle \min_{u\in K^*\atop \|u\|^*=1} \max_{v\in K \atop \|v\|=1} \ip{u}{v}}. 
\]
\item[(b)] If $\|\cdot\| = \|\cdot\|_2$ then
\[
1 \le \frac{\sigma(L)}{\nu(L)} \le \frac{1}{ \cos(\Theta(K^*,K))}.
\]
where 
\[
\Theta(K^*,K):=\max_{u\in K^*\setminus\{0\}}\min_{v\in K\setminus\{0\}} \angle(u,v).
\]
In particular, if $K^*\subseteq K$  then
$\nu(L) = \sigma(L).$
\item[(c)] If $\|\cdot\| = \|\cdot\|_e$ then
\[
\sigma(L) = \nu(L).
\] 

\end{description}
\end{proposition}
\begin{proof}
\begin{description}
\item[(a)] The first inequality is an immediate consequence of H\"older's inequality~\eqref{eq.holder}.  Next, from Proposition~\ref{prop.dual.sym} it follows that $\sigma(L) = \|\bar u - \bar y\|^*$ for some $\bar v \in K, \bar y\in L^\perp, \bar u\in K^*$ with $\|\bar v\| = 1, \ip{\bar u}{\bar v} = 1$.  Thus from the construction  of $\nu(L)$ we get
$$ 
\nu(L) \le \frac{\|\bar u - \bar y\|^*}{\|\bar u\|^*} \le \frac{\sigma(L)}{\displaystyle\min_{v\in K, u\in K^*\atop \|v\|=1,\ip{u}{v}=1}\|u\|^*}
$$ 
and hence the second inequality follows.

For the third inequality assume $\nu(L) = \|\hat u - \hat y\|^*$ for some $\hat u \in K^*, \hat y \in L^\perp$ with $\|\hat u\|^* = 1.$  Then by Proposition~\ref{prop.dual.sym} we get
\begin{align*}
\sigma(L)& = \min_{v\in K, y \in L^\perp, u\in K^* \atop \|v\|=1,\ip{u}{v}=1} \|u-y\|^* \leq \inf_{v\in K, y \in L^\perp, \atop \|v\|=1, \ip{\hat u}{v}\neq 0 } \left\|\frac{\hat u}{\ip{\hat u}{v}}-y\right\|^*\\
 & = \inf_{v\in K, y \in L^\perp, \atop \|v\|=1, \ip{\hat u}{v}\ne 0} \frac{\|\hat u-y\|^*}{\ip{\hat u}{v}}
 = \frac{\displaystyle\min_{y \in L^\perp} \|\hat u-y\|^*}{ \displaystyle\max_{v\in K\atop  \|v\|=1}  \ip{\hat u}{v}}\\
 & \leq  \frac{\|\hat u-\hat y\|^*}{\displaystyle \max_{v\in K\atop  \|v\|=1}  \ip{\hat u}{v}}. 
\end{align*}
Hence
\[
\sigma(L) \le \frac{\|\hat u - \hat y\|^*}{\displaystyle\max_{v\in K \atop \|v\|=1} \ip{\hat u}{v}} \le \frac{\nu(L)}{\displaystyle\min_{u\in K^*\atop \|u\|^*=1}\max_{v\in K \atop \|v\|=1} \ip{u}{v}}
\]
and the third inequality follows.
\item[(b)] The first inequality follows from part (a).  For the second inequality observe that since $\cos(\cdot)$ is decreasing in $[0,\pi]$
\begin{align*}
\cos(\Theta(K^*,K)) &= \min_{u\in K^*\setminus\{0\}}\max_{v\in K\setminus\{0\}} \cos(\angle(u,v)) \\
&= \min_{u\in K^*\setminus\{0\}}\max_{v\in K\setminus\{0\}} \frac{\ip{u}{v}}{\|u\|_2\cdot \|v\|_2} \\
&= \min_{u\in K^*\atop \|u\|_2 = 1}\max_{v\in K\atop \|v\|_2=1} \ip{u}{v}.
\end{align*}
The second inequality then follows from part (a) as well.

If in addition $K^*\subseteq K$ then $\Theta(K^*,K) = 0$ and consequently $\frac{\sigma(L)}{\nu(L)}=1$.

\item[(c)]  Since $\|\cdot\| = \|\cdot\|_e$, we have $\|e\| =1$ and $\|u\|^* = \ip{u}{e}$  for all $u\in K^*$.  Thus
$\displaystyle\min_{u\in K^*\atop \|u\|^*=1}\max_{v\in K \atop \|v\|=1} \ip{u}{v} \ge \displaystyle\min_{u\in K^*\atop \|u\|^*=1} \ip{u}{e} = 1.$
Therefore from part (a) it follows that $\frac{\sigma(L)}{\nu(L)} = 1.$ 
\end{description}

\end{proof}

The following example shows that the upper bound in Proposition~\ref{corol.nu.sigma}(b) is tight.  

\begin{example}\label{eg}  Let $E = \R^2$ be endowed with the dot inner product and let $K:=\{(x_1,x_2)\in E: \sin(\phi)x_2 \ge \cos(\phi)|x_1|\}$ where $\phi \in (0,\pi/2),$  $L = \{(x_1,x_2)\in E: x_1 = 0\}$, and $\|\cdot\| = \|\cdot\|_2$.
Then $K^*=\{(x_1,x_2)\in E: \cos(\phi) x_2 \ge \sin(\phi)|x_1|\}$ and $\nu(L) = \sin(\phi)$.  If $\phi \in (0,\pi/4)$ then $\sigma(L) = 1/(2\cos(\phi))$ and $\Theta(K,K^*) = \pi/2-2\phi$.  Hence for $\phi \in (0,\pi/4)$
  \[
\frac{\sigma(L)}{\nu(L)} = \frac{1}{2\sin(\phi)\cos(\phi)} = \frac{1}{\sin(2\phi)} = \frac{1}{\cos(\pi/2-2\phi)} = \frac{1}{\cos(\Theta(K,K^*))}.
\]  
On the other hand, if $\phi \in [\pi/4,\pi/2)$ then $\sigma(L) = \sin(\phi) = \nu(L),$ and $\Theta(K,K^*) = 0.$
\end{example}

\section{Symmetry measure}
\label{sec.symmetry.meas}
Next, we will consider a symmetry measure that has been used as a measure of conditioning~\cite{BellF08,BellF09b}.  This measure is defined as follows.   Given a set $S$ in a vector space such that $0\in S$, define
\begin{equation}\label{eq.sym.point}
\sym(0,S):= \max\{t\ge 0: w\in S \Rightarrow -tw \in S\}.
\end{equation}
Observe that $\sym(0,S) \in [0,1]$ with $\sym(0,S)=1$ precisely when $S$ is perfectly symmetric around $0$.  Furthermore, it is easy to see that
\begin{equation}\label{eq.sym.point.2}
\sym(0,S) = \min_{v\in S} \max_{x\in S} \{ t \ge 0: x+tv = 0\}.
\end{equation}
Define the analogous {\em symmetry} measure of the cone $K$ around the linear subspace $L$ as    
\begin{equation}\label{eq.sym.set}
\Sym(L,K):= \min_{v\in K\atop \|v\|\le 1} \max_{x\in K \atop \|x\|\le 1} \{ t \ge 0: x+tv\in L\}.
\end{equation}
The following proposition shows the equivalence between $\Sym(L,K)$ and the symmetry measure defined in~\cite{BellF08,BellF09b}.

\begin{proposition}\label{prop.symmetry} let $L\subseteq E$ be a linear subspace and $A:E\rightarrow F$ be such that $\ker(A) = L$.  Then
\[
\Sym(L,K)=\sym(0,S)
\]
for $S:=\{Ax: x\in K, \|x\|\le 1\}.$
\end{proposition}
\begin{proof}
The construction of $S$ together with~\eqref{eq.sym.point.2} and~\eqref{eq.sym.set} imply that 
\begin{align*}
\sym(0,S) &= \min_{v\in K\atop \|v\|\le 1} \max_{x\in K \atop \|x\|\le 1}\{t\ge 0: Ax+tAv = 0\} \\
&= 
\min_{v\in K\atop \|v\|\le 1} \max_{x\in K \atop \|x\|\le 1}\{t\ge 0: x+tv \in L\}\\
& = \Sym(L,K).
\end{align*}
\end{proof}

Observe that $L \cap \interior(K) \ne \emptyset$ if and only if $\Sym(L,K) > 0$. It is also easy to see that $\Sym(L,K) \in [0,1]$ for any linear subspace $L$ and $\Sym(L,K) = 1$ precisely when $K$ is perfectly symmetric around $L$ in the following sense: for all $v\in K$ there exists $x\in K$ such that $x+v\in L$ and $\|x\|\le \|v\|$.

The following result relating the symmetry and sigma measures is a general version of \cite[Proposition 22]{EpelF02}.

\begin{theorem} 
\label{thm.symmetry}
Let $L\subseteq E$ be a linear subspace such that $L \cap \interior K \neq \emptyset$.  Then
\[
\frac{\Sym(L,K)}{1+\Sym(L,K)}  \le \sigma(L) \le \frac{\Sym(L,K)}{1-\Sym(L,K)},
\]
with the convention that the right-most expression above is $+\infty$ if $\Sym(L,K) = 1$.
If there exists  $e\in \interior(K^*)$ such that $\|z\| = \ip{e}{z}$ for all $z\in K$  then
\[
\frac{\Sym(L,K)}{1+\Sym(L,K)} = \sigma(L).
\]
\end{theorem}
\begin{proof}
To ease notation, let $s:= \Sym(L,K)$ and $\sigma := \sigma(L)$.  First we show that $\sigma\ge \frac{s}{1+s}$.  To that end, suppose $v\in K, \|v\|=1$ is fixed.  The construction~\eqref{eq.sym.set} implies that there exists $z\in K, \|z\|\le 1$ such that $z+sv \in L$. Observe that $z+sv \ne 0$ because $z,v\in K$ are non-zero and $s\ge 0.$  Thus $x:=\frac{1}{\|z+sv\|}(z+sv)\in L, \|x\|=1$ and
\[
\lambda_{v}(x) \ge \frac{s}{\|z+sv\|} \ge \frac{s}{\|z\|+s\|v\|} \ge \frac{s}{1+s}.
\]  
Since this holds for any $v\in K, \|v\|=1$, it follows that
$\sigma\ge \frac{s}{1+s}.$

Next we show that $\sigma\le\frac{s}{1-s}$.  Assume $s < 1$ as otherwise there is nothing to show.  Let $v\in K, \|v\| = 1$ be such that 
\begin{equation}\label{eq.sym.less.one}
\max_{x\in K\atop \|x\|\le 1}\{t \ge 0: x+tv \in L\} < 1.
\end{equation}
At least one such $v$ exists because $s = \Sym(L,K) < 1$.

It follows from the construction of $\sigma(L)$ that there exists $x \in L, \|x\| = 1$ such that $\lambda_v(x)\ge \sigma > 0$.  In particular, $x - \sigma v \in K$.  Furthermore, $x-\sigma v \ne 0$ as otherwise $v =\frac{1}{\sigma} x\in L$ and  $x+v \in L$ which would contradict \eqref{eq.sym.less.one}.  Thus $z:=\frac{x-\sigma v}{\|x-\sigma v\|} \in K, \|z\|=1$ and 
$z +\frac{\sigma}{\|x-\sigma v\|} v \in L$ with
$\frac{\sigma}{\|x-\sigma v\|} \ge \frac{\sigma}{1+\sigma}.$  Since this holds for any $v\in K, \|v\|=  1$ satisfying ~\eqref{eq.sym.less.one}, it follows that $s \ge \frac{\sigma}{1+\sigma}$ or equivalently $\sigma \le \frac{s}{1-s}$.

Next consider the special case when there exists $e \in \interior(K^*)$ such that $\|z\| = \ip{e}{z}$ for all $z\in K$.  In this case,
$\|x-\sigma v \| = \ip{e}{x-\sigma v} = \ip{e}{x}-\ip{e}{\sigma v} = \|x\| - \sigma \|v\| = 1-\sigma$ in the previous paragraph and so the second inequality can be sharpened to $s \ge \frac{\sigma}{1-\sigma}$ or equivalently $\sigma \le \frac{s}{1+s}$.
\end{proof}

We also have the following relationship between the distance to infeasibility and the symmetry measure. 

\begin{corollary} 
\label{corol.symmetry} Let $L\subseteq E$ be a linear subspace such that $L \cap \interior(K) \neq \emptyset$.  Then
\[
\min_{u\in K^*\atop \|u\|^*=1} \max_{v\in K \atop \|v\|=1} \ip{u}{v}\cdot \frac{ \Sym(L,K)}{1+\Sym(L,K)} \le \nu(L) \le \frac{\Sym(L,K)}{1-\Sym(L,K)}.
\]
In particular, if $\|\cdot\| = \|\cdot \|_2$ then 
\[
\cos(\Theta(K^*,K))\cdot \frac{ \Sym(L,K)}{1+\Sym(L,K)} \le \nu(L) \le \frac{\Sym(L,K)}{1-\Sym(L,K)}.
\]
\end{corollary}
\begin{proof} This is an immediate consequence of Proposition~\ref{corol.nu.sigma} and  Theorem~\ref{thm.symmetry}.
\end{proof}

\section{Variants $\overline \nu(L)$ and $\V(L)$ of $\nu(L)$}
\label{sec.extend}

Consider the following variant of $\nu(L)$ that places the normalizing constraint on $y\in L^\perp$ instead of $u\in K^*$:
\[
\overline \nu(L):=\min_{\substack{u\in K^*, y\in L^\perp\\\|y\|^*=1}} \|y-u\|^*.
\]
It is easy to see that $\overline \nu(L) = \nu(L) = \sin\angle(L^\perp,K^*)$ when $\|\cdot\|=\|\cdot\|_2$.  However,  
$\overline \nu(L)$ and $\nu(L)$ are not necessarily the same for other norms.  This fact highlights one of the interesting nuances of non-Euclidean norms.

Like $\nu(L)$, its variant $\overline \nu(L)$ is closely related to Renegar's distance to infeasibility as stated in Proposition~\ref{thm.renegar.again} below. 
Suppose $A:E \rightarrow F$ is a linear mapping and
consider the conic systems \eqref{primal} and \eqref{dual} defined by taking $L = \ker(A)$, that is,
\begin{equation}\label{renegar.primal.again}
 Ax = 0, \; x\in K  \setminus\{0\},
\end{equation}
and \begin{equation}\label{renegar.dual.again} 
A^*w \in K^* \setminus\{0\}.
\end{equation}
In analogy to $\dist(A,\I)$, define $\overline \dist(A,\I)$ as follows
\begin{align*}
\overline\dist(A,\I) &:= \inf\left\{\|A - \tilde A\|: \tilde Ax = 0, x \in K\setminus\{0\} \; \text{ is infeasible}\right\}
\\ 
& \; = \inf\left\{\|A - \tilde A\|: \tilde A^* w \in K^* \; \text{ for some } w \in F\setminus\{0\}\right\}.
\end{align*}

A straightforward modification of the proof of Theorem~\ref{thm.renegar} yields Proposition~\ref{thm.renegar.again}.  We note that this proposition requires  that $A$ be surjective.  This is necessary because $\overline \dist(A,\I) = 0$ whenever $A$ is not surjective whereas $\|A\|,\|A^{-1}\|,$ and  $\overline \nu(L)$ may all be positive and finite.  The surjectivity of $A$ can be evidently dropped if the definition of $\overline \dist(A,\I)$ is amended by requiring $\Image(\tilde A) = \Image(A)$.

\begin{proposition}\label{thm.renegar.again} 
Let $A \in \mathcal L(E,F)$ be a surjective linear mapping such that \eqref{renegar.primal.again} is strictly feasible and let $L:=\ker(A)$.  Then
\[
\frac{1}{\|A\|} \le \frac{\overline\nu(L)}{\overline\dist(A,\I)} \le \|A^{-1}\|.
\]
\end{proposition}
\begin{proof}
First, we prove $\overline\dist(A,\I) \le \overline\nu(L) \|A\|.$  To that end, let $\bar y \in L^\perp$ and $\bar u \in K^*$ be such that $\|\bar y\|^* =1$ and 
$
\overline\nu(L) = \|\bar y - \bar u\|^*$. Since $\bar y \in L^\perp = \Image(A^*)$ and $\|\bar y\|^* =1$, it follows that $\bar y = A^* \bar v$ for some $\bar v\in F$ with $|\bar v|^* \ge 1/\|A\|$.  Let $\bar z\in F$ be such that $|\bar z| = 1$ and $\ip{\bar v}{\bar z} = |\bar v|^* =1$.  
Now construct $\Delta A: E \rightarrow F$ as follows
\[
\Delta A(x):= \frac{\ip{\bar u-\bar y}{x}}{|\bar v|^*} \bar z.
\]
Observe that $\|\Delta A\| =  \|\bar y - \bar u\|^*/|\bar v|^* \le \nu(L)\|A\|,$ and $\Delta A^*: F \rightarrow E$ is defined by
\[
\Delta A^*(w) = \frac{\ip{w}{\bar z}}{|\bar v|^*} (\bar u - \bar y).
\]
In particular $(A+ \Delta A)^*\bar v = A^* \bar v + (\bar u - \bar y) = \bar u \in K^*$ and $\bar v \in F\setminus \{0\}.$  Therefore
\[
\overline \dist(A,\I) \le \|\Delta A\| \le \nu(L)\|A\|.
\]
Next, we prove $\overline\nu(L) \le \|A^{-1}\| \cdot \overline\dist(A,\I)$. To that end, suppose  
$\tilde A\in \mathcal L(E,F)$ is such that $\tilde A^* \bar w \in K^*$ for some $\bar w\in F\setminus\{0\}.$  Since $A$ is surjective, $A^*$ is one-to-one and thus $A^* \bar w \ne 0$.  Without loss of generality{\color{blue},} we may assume that $\|A^*\bar w\|^*=1$ and so $|\bar w|^* \le \|A^{-1}\|$.   
It thus follows that
\[
\overline\nu(L) \le \min_{u\in K^*}\|A^*\bar w - u\| \le \|A^*\bar w - \tilde A^* \bar w\|^*
\le \|A^{-1}\|\cdot \|\tilde A - A\|. 
\]
Since this holds for all $\tilde A\in \mathcal L(E,F)$ such that $\tilde A^* w \in K^*$ for some $w \in F\setminus\{0\}$, it follows that 
\[
\overline\nu(L) \le \|A^{-1}\| \cdot \overline\dist(A,\I).
\]
\end{proof}

\bigskip

Next, consider an extension $\V(L)$ of $\nu(L)$ obtained by de-coupling the normalizing constraint of $u \in K^*$ from the norm defining its distance to $L^\perp$.  More precisely, suppose $\vertiii{\cdot}$ is an additional norm in the space $E$ and consider the following extension of $\nu(L)$
\[
\V(L):=\min_{\substack{u\in K^*, y\in L^\perp\\\|u\|^*=1}} \vertiii{y-u}^*.
\]
Proceeding as in Proposition~\ref{prop.nu}, it is easy to see that 
$
\V(L)=\displaystyle\min_{u\in K^* \atop \|u\|^*=1} \max_{x\in L\atop \vertiii{x}\le 1} \ip{u}{x}.
$
Thus only the restriction of $\vertiii{\cdot}$ to $L$ matters for $\V(L)$.    The following proposition considers a special case when this additional flexibility is particularly interesting.

\begin{proposition}  Suppose $L = \Image(A)$ for some linear map $A:F\rightarrow E$. Define the norm $\vertiii{\cdot}$ in $L$ as follows
\begin{equation}\label{eq.norm.A}
\vertiii{x}:=\min_{w\in A^{-1}(x)} |w|,
\end{equation}
where $|\cdot|$ denotes the norm in $F$.  Then
$$\V(L) = \dist(A,\I).$$
\end{proposition}
\begin{proof}  This follows via a straightforward tweak of the proof of Theorem~\ref{thm.renegar}.  
\end{proof}

\medskip

The additional flexibility of $\V(L)$ also yields the following extension of
Proposition~\ref{prop.eigenvalue}:  If $\|\cdot\| = \|\cdot\|_e$ for some $e\in \interior(K)$ then for any linear subspace $L\subseteq E$ and any additional norm $\vertiii{\cdot}$ in $L$
\[
\V(L) = \max_{x\in L\atop \vertiii{x} \le 1}  \lambda_{e}(x). 
\]

The construction of $\sigma(L)$ can be extended in a similar fashion  by de-coupling the normalizing constraints of $v \in K$ and $x \in L$.     More precisely, let $\vertiii{\cdot}$ be an additional norm in $L$ and consider the following extension of $\sigma(L)$:
\[
\Sigma(L) := \min_{v\in K\atop \|v\|=1}  \max_{x\in L \atop \vertiii{x} \le 1}\lambda_v(x).
\]

The additional flexibility of $\Sigma(L)$ readily yields the extension of Proposition~\ref{corol.nu.sigma} to the more general case where $\nu(L)$ and $\sigma(L)$ are replaced with $\V(L)$ and $\Sigma(L)$ respectively for any additional norm $\vertiii{\cdot}$ in $L$.

\bigskip

Finally, consider the extension $\overline \V(L)$ of $\overline \nu(L)$ obtained by de-coupling the normalizing constraint of $y \in L^\perp$ from the norm defining its distance to $ K^*$.
Suppose $\vertiii{\cdot}$ is an additional norm in the space $L^\perp$ and consider the following extension of $\overline\nu(L)$:
\[
\overline\V(L):=\min_{\substack{u\in K^*, y\in L^\perp\\ \vertiii{y}^*=1}} \|y-u\|^*.
\]
To illustrate the additional flexibility of $\overline \V(L)$ consider the special case when $L = \ker(A)$ for some surjective linear mapping $A:E \rightarrow F$ and define the norm $\vertiii{\cdot}$ in $L^\perp$ as follows
\begin{equation}\label{eq.norm.A.again}
\vertiii{x}:=|Ax|,
\end{equation}
where $|\cdot|$ denotes the norm in $F$.  A straightforward tweak of the proof of Proposition~\ref{thm.renegar.again} shows that $\overline\V(L) = \overline\dist(A,\I)$ for this choice of norm.   

\section{Conclusion}

We propose an approach to integrate a variety of proposed condition 
 measures for a homogeneous conic system of the form
 \[
 \text{ find} \; x\in L\cap K \setminus\{0\}.
 \]  
 Our approach hinges on the following concept of {\em data-independent distance to infeasibility}:  
\[
\nu(L):=\min_{u\in K^*, y\in L^\perp\atop\|u\|^*=1} \|u-y\|^*.
\] 
This quantity is  based solely on the following three minimal components associated to a linear conic system:  the cone $K$, linear subspace $L$, and some underlying norm $\|\cdot\|$ in the ambient space.
 
The data-independent distance to infeasibility $\nu(L)$ is a non-Euclidean generalization of the Grassmannian condition measure introduced by Belloni and Freund~\cite{BellF09a}, and further extended by Amelunxen and B\"urgisser~\cite{AmelB12}.  The non-Euclidean flexibility allows us to establish a number of novel and interesting relationships among several popular condition measures whose exact relationship with each other was not fully understood before.  These measures include our new data-independent distance to infeasibility, Renegar's data-dependent condition measure, the Grassmanian condition measure, a measure of symmetry, a measure of most interior solution,  and a measure of depth.  The latter two measures are constructed via some canonical {\em induced eigenvalue mappings} and {\em induced norm} that feature key structural properties of the underlying cone.
 
Our main results provide valuable insight into the tradeoffs of different notions of conditioning and thus pave the road for improved algorithmic developments that are more effectively adept to the intrinsic difficulty of a problem instance.  In particular, our results readily suggest preconditioning and reconditioning techniques like those that underlie a variety of recent {\em rescaling} algorithms.

The following two natural variants of $\nu(L)$ offer additional flexibility and enable a tighter integration among different condition measures. The first one places the normalization on $y\in L^\perp$ instead of $u\in K^*$:
\[
\overline \nu(L):=\min_{u\in K^*, y\in L^\perp\atop\|y\|^*=1} \|u-y\|^*.
\]  
The second one adds a dimension of flexibility by allowing the use of different norms for the normalization of $u\in K^*$ and the difference $u-y$: 
\[
\V(L):=\min_{u\in K^*, y\in L^\perp\atop\|u\|^*=1} \vertiii{u-y}^*.
\]

\section*{Acknowledgements}

We are grateful to two anonymous referees for numerous suggestions on a previous version of this paper.  In particular, we thank one of the anonymous referees for suggesting the natural definition~\eqref{eq.sym.set} of $\Sym(L,K)$ in Section~\ref{sec.symmetry.meas}.

Javier Pe\~na's research has been funded by NSF grant CMMI-1534850.
Vera Roshchina is grateful to the Australian Research Council for continuous support via grants DE150100240 and DP180100602. 

\bibliographystyle{plain}

\end{document}